\documentclass[12pt, a4paper]{amsart}

\usepackage{fullpage}

\usepackage{color}
\definecolor{webcolor}{rgb}{0,0,1}
\definecolor{webbrown}{rgb}{.6,0,0}
\usepackage[
        colorlinks,
        linkcolor=webbrown, filecolor=webbrown,  citecolor=webbrown,
        pdfauthor={},
  pdftitle={},
]{hyperref}

\usepackage{amsmath}
\usepackage{amsthm}
\usepackage{amssymb}
\usepackage[capitalise]{cleveref}
\usepackage{booktabs}

\usepackage{tikz-cd}

\def\Z{\mathbf Z}
\def\F{\mathbf F}
\def\Q{\mathbf Q}

\def\Pfactor{\textsc{Factor}}
\def\Pdisclog{\textsc{DiscLog}}
\def\Pprimroots{\textsc{PrimitiveRoot}}
\DeclareMathOperator{\mat}{M}
\DeclareMathOperator{\rank}{rank}
\DeclareMathOperator{\GL}{GL}
\DeclareMathOperator{\SL}{SL}

\DeclareMathOperator{\jac}{J}
\DeclareMathOperator{\K1}{K_1}
\DeclareMathOperator{\Aut}{Aut}
\DeclareMathOperator{\image}{im}

\theoremstyle{definition}
\newtheorem{theorem}{Theorem}[section]
\newtheorem{lemma}[theorem]{Lemma}
\newtheorem{example}[theorem]{Example}
\newtheorem{remark}[theorem]{Remark}
\newtheorem{cor}[theorem]{Corollary}
\newtheorem{prop}[theorem]{Proposition}
\newtheorem*{problem*}{Problem}

\Crefname{cor}{Corollary}{Corollaries}

\newtheorem{introtheorem}{Theorem}[section]

\newcommand{\ab}[1]{#1^{\times \mathrm{ab}}}

\title{Determining unit groups and $\mathrm{K}_1$ of finite rings}

\author{Tommy Hofmann}
\address{
Naturwissenschaftlich-Technische Fakult\"at\\
Universit\"at Siegen\\
Walter-Flex-Straße 3\\
57068 Siegen\\
Germany}
\email{tommy.hofmann@uni-siegen.de}

\date{\today}



\begin{document}

\renewcommand{\theenumi}{(\roman{enumi})}%

\begin{abstract}
  We consider the computational problem of determining the unit group of a finite ring, by which we mean
  the computation of a finite presentation together with an algorithm to express units as words in the generators.
  We show that the problem is equivalent to the number theoretic problems of factoring integers and solving discrete logarithms in finite fields.
  A similar equivalence is shown for the problem of determining the abelianization of the unit group or the first $K$-group of finite rings.
\end{abstract}

\maketitle

\section{Introduction}

Throughout, all rings are assumed to be associative unitary rings.
We consider the computational problem of determining the structure of the unit group of a finite ring $R$.
Informally, by this we mean determining a finitely presented group $P = \langle X \, | \, T \rangle$ and an isomorphism
$\alpha \colon R^\times \to P$
such that both $\alpha$ and $\alpha^{-1}$ can be evaluated efficiently (see Section~\ref{sec:prelim} for precise definitions regarding the encoding of objects and complexity).

Limitations of such an undertaking are immediate once one considers the situation for commutative rings,
 for which computations with and in unit groups are a classical topic in algorithmic number theory and cryptography.
 For example, determining the order $\lvert (\Z/n\Z)^\times \rvert$ of the unit group of the finite ring $\Z/n\Z$ is equivalent to the problem of factoring $n$; see~\cite{MR480295,MR2488898}.
For a finite field $\F_q$, determining a generator of the unit group amounts to finding a primitive root (which reduces to factoring $q - 1$), while expressing units with respect to a generator is the discrete logarithm problem for finite fields.
Given these observations, it is reasonable to expect that determining unit groups of arbitrary finite rings is related to these ``number theoretic obstacles'', the problems of factoring integers (\Pfactor) and of computing discrete logarithms in finite fields (\Pdisclog).

Another class of finite commutative rings, where algorithms for computing unit groups have been investigated extensively, are quotient rings of the form $\mathcal{O}/I$, where $I$ is a nonzero ideal of a $\Z$-order $\mathcal{O}$ in a number field; see~\cite{MR1972751, MR1443117, MR1446497}.
These have been generalized in~\cite{MR2166795} to quotients of $\F_q[X]$-orders in global function fields and in~\cite{MR2123128} to $\Z$-orders in étale $\Q$-algebras.
Although not explicitly stated, it is straightforward to show that these algorithms reduce the computation of and with these unit groups to \Pfactor{} and \Pdisclog.
In~\cite{MR4782659}, the authors consider the computation of unit groups of finite commutative $\Z$-algebras, which include finite commutative rings as a special case.

Similar to the commutative case, unit groups of finite noncommutative rings play an important role when investigating $\Z$-orders in noncommutative (semisimple) $\Q$-algebras. %
For such an order $\Lambda$, Bley--Boltje~\cite{MR2282916} describe an algorithm to determine generators for certain quotients $(\Lambda/I)^\times$ (later generalized in~\cite{MR4493243} and shown to reduce to \Pfactor), a subproblem when determining the locally free class group of $\Lambda$.
In~\cite{bley2024determinationstablyfreecancellation}, determination of the stably free cancellation property for $\Z$-orders is studied, where maps between unit groups $(\Lambda/I)^\times$ play an important role. Note that in loc. cit. this is solved in a naive way by embedding the unit group into a symmetric group or a matrix group over a finite field if possible (see~\Cref{sec:computations} for a comparison).
Our result on unit groups of finite rings will include all of the aforementioned results as special cases (see~\Cref{cor:arithmetic}).

For certain classes of finite rings the unit group has a natural representation as a specific type of finite group, for example as a finitely generated abelian group in case of commutative rings or as a linear group over a finite field for $\F_p$-algebras.
Since for arbitrary finite rings, there does not seem to be a natural representation of the unit group apart from the (prohibitively expensive) regular permutation representation, we determine the unit group as a finitely presented group.
More precisely, we determine an effective presentation of the unit group, which includes both a finitely presented group and an algorithm to express units as words in the generators (see \Cref{sec:prelim} for the precise definition).
The first aim of this note is to show that for arbitrary finite rings the aforementioned ``number theoretic obstacles'' from the commutative case are the only obstacles to finding effective presentations (see~\Cref{thm:mainmain}):

\begin{introtheorem}\label{thm:intro1}
  The following computational problems are probabilistic polynomial-time reducible to each other:
\begin{enumerate}
  \item
    \Pfactor{} and \Pdisclog,
  \item
    computing effective presentations of unit groups of finite rings.
\end{enumerate}
\end{introtheorem}

(Variants of \Cref{thm:intro1} for finite quotients of arithmetic orders and finite-dimensional algebras over finite fields are described in~\Cref{cor:alg,cor:arithmetic}.)

The reduction from (i) to (ii) follows from a ``folklore'' probabilistic variant of an algorithm of Miller~\cite{MR480295}; see~\cite[10.4 \& 10.5]{MR2488898}.
Our main contribution is \Cref{thm:main}, where we prove that an effective presentation of the unit group of a finite ring $R$ can be determined in probabilistic polynomial time, once oracles for solving \Pfactor{} and \Pdisclog{} are available.
We give an overview of the reduction,
which is very much inspired by Bley--Boltje~\cite{MR2282916}, where an algorithm for determining generators of unit groups of finite quotient rings of $\Z$-orders is given.
Let $J$ be the Jacobson radical of $R$.
There is a canonical exact sequence 
\[ 1 \longrightarrow 1 + J \longrightarrow R^\times \longrightarrow (R/J)^\times \longrightarrow 1, \]
which shows that the unit group is an extension of the units of the semisimple ring $R/J$ (a product of matrix rings over finite fields) by $1 + J$, the unipotent elements of $R$.
The group $1 + J$ is solvable with abelian series defined by $1 + J^{2^i}$ and factors isomorphic to $J^{2^i}/J^{2{^i + 1}}$, $i \in \Z_{\geq 0}$.
To determine a finite presentation and to express elements in terms of the generators, we investigate the corresponding problems for $1 + J$ (using the aforementioned abelian series) and $(R/J)^\times$. This is then combined using a general extension
result for effective presentations of black box groups (\Cref{thm:patching}).

As an application, for a finite ring $R$ we consider the problem of determining $\ab{R}$ and the first K-group $\K1(R)$.
(For a group $G$, denote by $G^{\mathrm{ab}} = G/[G, G]$ the abelianization, that is, the maximal abelian quotient, of $G$.)
The construction of the first K-group goes back to work of Whitehead~\cite{MR35437} on simple homotopy types and was later formulated for rings by Bass--Samuel~\cite{MR152559}.
The abelian group $\K1(R)$ is by definition the abelianization $\GL(R)^{\mathrm{ab}}$ of the infinite general linear group $\varinjlim \GL_n(R)$ of $R$.
Both the abelianization of unit groups and the first K-group of finite rings play an important role in the study of noncommutative $\Z$-orders.
Applications include determining (reduced) projective class groups of integral group rings $\Z[G]$ (\cite{MR703486}) and the stably free cancellation property of $\Z$-orders (\cite{bley2024determinationstablyfreecancellation}).
One can deduce from Theorem~\ref{thm:intro1} that the computational problem of determining these invariants is as hard as computing unit groups themselves (see~\Cref{thm:mainmain}):

\begin{introtheorem}\label{thm:intro2}
  The following computational problems are probabilistic polynomial-time reducible to each other:
\begin{enumerate}
  \item
    computing effective presentations of unit groups of finite rings,
  \item
    computing abelianizations of unit groups of finite rings,
  \item
    computing $\K1$ of finite rings.
\end{enumerate}
\end{introtheorem}

Our interest in algorithms for unit groups of finite rings is motivated by the investigation of noncommutative $\Z$-orders.
To this end, we have implemented the algorithms of this paper in the computer algebra systems \textsc{Hecke}~\cite{FHHJ2017} and~\textsc{OSCAR}~\cite{OSCAR, OSCAR-book}.
Although \Pfactor{} and \Pdisclog{} are considered to be hard, 
state-of-the-art computer algebra systems incorporate sophisticated implementations (with heuristic subexponential complexity) that are able to solve these problems for many parameter sizes of interest.
In combination with our reduction, this leads to an implementation for computing unit groups of finite rings, which performs quite well in practice, substantially improving upon the naive version.
For example, we have determined the unit groups and first K-groups of the group rings $\F_2[G]$ for all non-abelian groups of even order up to $32$.
See~\Cref{sec:computations} for details.

We end this introduction with an outline of the manuscript.
In Section~\ref{sec:prelim} we review the encoding of the objects involved, and introduce the notion of an effective presentation of a black box group.
Section~\ref{sec:ext} contains the main algorithmic tool on extensions of black box groups,
which is then applied to unit groups of finite rings in Section~\ref{sec:unit}.
The applications to abelianization and first K-group is given in Section~\ref{sec:abinv}.
In Section~\ref{sec:compproblems}, we use known results on finite commutative rings to obtain the ``reverse'' reductions in Theorems~\ref{thm:intro1} and~\ref{thm:intro2}.
Specializations to finite quotients of arithmetic orders, including orders of number fields and finite-dimensional algebras over finite fields, can be found in~\cref{sec:applications}.
Finally, details on the implementation and applications are given in Section~\ref{sec:computations}.

\subsection*{Acknowledgments}
The author wishes to thank Thomas Breuer, Claus Fieker and Max Horn for helpful discussions on computational group theory and and is grateful for numerous helpful comments and corrections from Werner Bley, Henri Johnston and Wilken Steiner.

\subsection*{Funding}
The author gratefully acknowledges support by the Deutsche Forschungsgemeinschaft —
Project-ID 286237555 – TRR 195; and Project-ID 539387714.

\section{Preliminaries}\label{sec:prelim}

We briefly recall the conventions that we will use for the complexity analysis of our algorithms. For details we refer the reader to Lenstra~\cite{Lenstra1992} or Cohen~\cite[\S{}1.1]{MR1228206}.
We are solely using bit complexity to analyze the running time of our algorithms. %
An algorithm is \textit{polynomial-time} if the runtime is in $l^{O(1)}$, where $l$ is the length of the input.
By a probabilistic algorithm we mean an algorithm that may call a random number generator. A probabilistic algorithm is polynomial time, if the expected running time is polynomial in the length of the input.
By abuse of notation, we will often speak of the length of an object $X$, by which we mean the length of the encoding of $X$.
We denote the problem of factoring integers as \Pfactor, the problem of computing a generator of the multiplicative group of a finite field as \Pprimroots{} and the problem of computing discrete logarithms in finite fields as \Pdisclog; see~\cite{MR1745660}.

\subsection{Abelian groups and finite rings}

Finite abelian groups will be encoded as a sequence of positive integers $d_1,\dotsc,d_n$, specifying the group $\Z/d_1\Z \times \dotsb \times \Z/d_n\Z$; and their elements as integers $x_1,\dotsc,x_n$ with $0 \leq x_i < d_i$ for $1 \leq i \leq n$.
Note that a finitely generated abelian group given by generators and relators can be transformed into this presentation in polynomial time by means of the Smith normal form~\cite{MR573842}.
Using this encoding of finite abelian groups, one can derive encodings for group homomorphisms, subgroups or quotient groups,
for which many computational tasks can be performed in deterministic polynomial time; see~\cite{Lenstra1992} and~\cite[Chapter 2]{ciocaneateodorescu:tel-01378003}.

The encoding of finite abelian groups and their morphisms extends to finite rings. A finite ring $R$ is encoded as a finite abelian group $R^+$ together with the multiplication map $m \colon R^+ \times R^+ \to R^+$.
Distributivity of multiplication over addition implies that $m$ is bilinear and thus induces a homomorphism $R^+ \otimes_{\Z} R^+ \to R^+$ of abelian groups.
By appealing to the underlying abelian group, we also obtain encodings for ideals, subrings, quotient rings and ring morphisms.
As in the case of finite abelian groups, many basic questions about these objects can be answered in deterministic polynomial time; see~\cite[Chapter 3]{ciocaneateodorescu:tel-01378003}.
Note that the encoding of a finite ring $R$ has length bounded by $\log(\lvert R \rvert)^3$.
Thus a (probabilistic) polynomial time algorithm which has input a finite ring $R$ is one with complexity in $\log(\lvert R \rvert)^{O(1)}$.

\subsection*{Black box groups and their presentations}

If $R$ is a finite ring, then the inclusion $R^\times \subseteq R$ provides a (unique) way to encode elements of the unit group. Moreover, multiplication of the ring provides us with a way to perform the group operation in $R^\times$ in polynomial time.
These properties allows us to consider $R^\times$ as a black box group, as introduced in~\cite{DBLP:conf/focs/BabaiS84}.
We will use the following modified definition of a black box group: 
A black box group $G$ consists of a set of elements with finite encoding of bounded length together with polynomial time algorithms to perform the group operation, inversion and to test equality.
Based on this, we can also represent maps between black box groups, for which we in addition require that one can determine preimages: If $f \colon G \to H$ is a map of black box groups and $h \in \image(f)$, one can find in polynomial time an element $g \in G$ with $f(g) = h$.

We are interested in computing finite presentations of black box groups. A finitely presented group $P = \langle X \,|\, T \rangle$ will be encoded by the number of generators $\lvert X \rvert$ and the set of relators $T$, where relators are encoded as elements of $(X \cup X^{-1})^*$, that is, as words over the alphabet $(X \cup X^{-1})$, with binary encoded exponents.
We thus use the ``bit-length'' of the relators, as opposed to the ``word length'', which amounts to using unary encoded exponents. Note that both encodings are equivalent from a complexity point of view (see~\cite[1.2]{MR2393425}), but in our context it is more convenient to work with binary encoded exponents.
An \textit{effective presentation} of a black box group $G$ consists of
\begin{itemize}
  \item a finitely presented group $P = \langle X  \,|\, T \rangle$,
  \item a probabilistic polynomial-time algorithm, that given $g \in G$ determines a word $\alpha(g) \in (X \cup X^{-1})^*$,
  \item a probabilistic polynomial-time algorithm, that given $w \in (X \cup X^{-1})^*$ determines an element $\alpha^{-1}(w) \in G$,
  \item such that $\alpha$ induces an isomorphism $G \to P, g \mapsto \alpha(g)$ with inverse induced by $w \mapsto \alpha^{-1}(w)$.
\end{itemize}
By abuse of language, we refer to $\alpha$ as the \textit{discrete logarithm} and $\alpha^{-1}$ as the \textit{exponential map} of the effective presentation.
Informally speaking, determining an effective presentation of a black box group $G$ amounts to finding a generating set of $G$ and a probabilistic polynomial-time algorithms to express arbitrary elements as words in those generators.
In the following, for a black box group $G$ we will speak of an effective presentation $\alpha \colon G \to P$ with the existence of the required algorithms implied.
The computational problem is now:

\begin{problem*}
  Given a finite ring $R$, determine an effective presentation $\alpha \colon R^\times \to \langle X \, | \, T \rangle$ of the unit group, considered as a black box group.
\end{problem*}

\begin{remark}
  Note that by just asking for an algorithm that, given a finite ring $R$, returns a presentation and an algorithm which expresses elements of $R^\times$ in terms of the generators, the problem can be phrased without the notion of black box groups.
  We introduce the notion of presentations of black box groups as a convenient bookkeeping device to keep track of the various objects, which behaves well with respect to extensions (see~\Cref{thm:patching}) and simplifies the exposition. 
  The final result (\Cref{thm:main}) itself is independent of this notion.
  For the same reason we use a stronger notion of black box groups then the original one from~\cite{DBLP:conf/focs/BabaiS84}. It is more convenient to work with and all the black box groups we consider satisfy the stronger assumptions.
\end{remark}

\begin{example}\label{example:blackbox}
  \hfill
  \begin{enumerate}
    \item
    For a finite ring $R$, the multiplicative group $R^\times$ (together with the encoding provided by $R^\times \subseteq R$) is a black box group.
    Now let $I$ be a two-sided nilpotent ideal of $R$.
    Then $1 + I \subseteq R^\times$ is a black box group.
    Also, if $J$ is another two-sided nilpotent ideal and $1 + J$ a normal subgroup of $1 + I$, then the factor group $(1 + I)/(1 + J)$ is a black box group (without unique encoding).
    \item
    Any finite abelian group $A = \Z/n_1\Z \times \dotsb \times \Z/n_r \Z$ can be regarded as a black box group.
    The group $A$ is isomorphic to a finitely presented group with $r$ generators and $r + r(r -1)/2$ relators,
    \[ A \longrightarrow \langle x_1, \dotsc,x_r \mid x_i^{n_i},\, [x_i,x_j], \text{ for $1 \leq i < j \leq r$} \rangle, \, (a_1,\dotsc,a_r) \mapsto x_1^{a_1} \dotsm x_r^{a_m}. \] 
    As both the discrete logarithm and exponential map can be evaluated in polynomial time, this is an effective presentation of $A$.
  \item
    Let $G = \F_q^\times$ be the multiplicative group of a finite field of order $q$, which is cyclic of order $q - 1$ and thus isomorphic to $\langle x \, | \, x^{q - 1} \rangle$.
    Turning this into an effective presentation reduces (by definition) in polynomial time to \Pprimroots{} and \Pdisclog.
  \end{enumerate}
\end{example}

\begin{lemma}\label{lem:slnk}
  The problem of determining an effective presentation of $\SL_n(k)$ given a finite field $k$ and $n \in \Z_{\geq 0}$ reduces in polynomial time to $\Pprimroots$ and $\Pdisclog$.
\end{lemma}

\begin{proof}
    We may assume that $n \geq 2$.
    The field $k$, a finite extension of some prime field $\F_p$ of degree $m$, is encoded as a simple extension $\F_p(\alpha)$ via a minimal polynomial in $\F_p[T]$, and thus of length $m \cdot \log(p) = \log(\lvert k \rvert)$.
    For $i, j \in \{1,\dotsc,n\}$ with $i \neq j$ and $\beta \in k$ denote by $e_{ij}(\beta)$ the elementary matrix with $1$'s on the diagonal and $\beta$ at position $(i, j)$.
    Since $k$ is a field, generated as an abelian group by $1,\alpha,\dotsc,\alpha^{m-1}$, $\SL_n(k)$ is generated by the elementary matrices $e_{ij}(\alpha^r)$, $i \neq j$, $0 \leq r < m$.
    Also, by applying row and column operations, there exists a polynomial-time algorithm that expresses an element $\SL_n(k)$ as a word in these generators.

  For $n \geq 3$, we can now use a result of~\cite{MR1871620}, where it is proven that $\SL_n(k)$ has a presentation corresponding to the generating set $e_{ij}(\alpha^r)$, with at most $4 n^4 m^3$ relators, which can be determined in polynomial time.

  For $n = 2$, we consider the presentations from~\cite[Sec. 3]{MR4044701} (which are derived from~\cite{MR1033184} using a different generating set): Determine a primitive element $\omega \in k^\times$ (using \Pprimroots{}) and let
    \[ \tau = \begin{pmatrix} 1 & 1 \\ 0 & 1 \end{pmatrix}, \, \delta = \begin{pmatrix} \omega^{-1} & 0 \\ 0 & \omega \end{pmatrix}, \, U = \begin{pmatrix} 0 & -1 \\ 1 & 0 \end{pmatrix} \in \SL_2(k). \] 
    Then $\SL_2(k)$ is generated by $\tau$, $\delta$ and $U$. Moreover, $\SL_2(k)$ has a presentation corresponding to those generators with at most 10 relators, which can be determined by computing discrete logarithms in $k^\times$ with respect to $\omega$ or $\omega^2$ (using \Pdisclog{}).
    To turn this into an effective presentation of $\SL_2(k)$, it suffices to express $e_{ij}(\alpha^r)$ in terms of $\tau$, $\delta$ and $U$.
    As
    $\delta^{-1}\tau \delta^{-1} = e_{12}(\omega^2)$ and $e_{21}(1) = u^{-1}\tau^{-1}u$,
    this can be done by computing discrete logarithms in $k^\times$ (using \Pdisclog{}).
\end{proof}

\begin{remark}
  In the proof of~\Cref{lem:slnk} we have made use of the (modified) Steinberg presentation (in case $n \geq 3$) and not other known presentations, like the ones presented in~\cite{MR1033184, MR2393425, MR2746771, MR4044701}, which are often ``shorter''.
  This is mostly due to the fact that for the Steinberg presentation it is straightforward to express arbitrary elements as words in the generators.
  For the presentations from~\cite{MR4044701}, which are closely related to the constructive matrix group recognition problem, it is also known how to express elements in terms of the generators, but these results employ the notion of straight-line programs (instead of words) and are thus not directly applicable in our setting.
  Using the Steinberg presentation also allows for the following curious result: 
  For $n \neq 2$, neither \Pprimroots{} nor \Pdisclog{} were needed in the proof of~\Cref{lem:slnk}.
  Thus, there exists a polynomial-time algorithm that given a finite field $k$ and $n \in \Z_{\geq 0}$, $n \neq 2$, determines an effective presentation of $\SL_n(k)$.
\end{remark}

\section{Extensions of effective presentations}\label{sec:ext}

The black box groups we consider and for which we want to determine effective presentations are usually extensions of smaller black box groups.
To use this algorithmically, in this section
we show that black box groups with effective presentations are closed under extensions.
To this end, we extend the well-known construction of presentations for extensions of finitely presented groups, as found for example in~\cite{MR2129747} or~\cite{MR1472735}.

\begin{theorem}\label{thm:patching}
  There exists an algorithm, that given an exact sequence
  \[ 1 \longrightarrow N \xrightarrow{\,\,\,\iota\,\,\,}{} G \xrightarrow{\,\,\,\pi\,\,\,}{} H \longrightarrow 1 \] 
  of black box groups and effective presentations of $N$ and $H$, determines an effective presentation of $G$ in polynomial time.
  Moreover, if $H$ has presentation $P_H = \langle X \, | \, T \rangle$ and $N$ presentation $P_N = \langle Y \, | \, U \rangle$,
  then the presentation $P_G = \langle Z \, | \, V \rangle$ constructed by the algorithm satisfies
  \[ \lvert Z \rvert \leq \lvert X \rvert + \lvert Y \rvert \text{ and } \lvert V \rvert \leq \lvert U \rvert + \lvert T \rvert + \lvert X \rvert \cdot \lvert Y \rvert. \] 
\end{theorem}

\begin{proof}
  Let $\alpha_N \colon N \to P_N$ and $\alpha_H \colon H \to P_H$ be the given effective presentations.
  We will construct an effective presentation $\alpha_G \colon G \to P_G$ and maps $\iota'$, $\pi'$ such that the diagram
  \[ 
\begin{tikzcd}
1 \arrow[r] & N \arrow[r, "\iota"] \arrow[d, "\alpha_N"]  & G \arrow[r, "\pi"] \arrow[d, dotted, "\alpha_G"] & H \arrow[r] \arrow[d, "\alpha_H"] & 1 \\
1 \arrow[r] & P_N \arrow[r, "\iota'", dotted] & P_G \arrow[r, "\pi'", dotted]        & P_H \arrow[r]         & 1 
\end{tikzcd}
\]
  commutes.
  We let $Z = X \sqcup Y$ be the disjoint union of $X$ and $Y$, which we consider as subsets of $Z$.
  For each $x \in X$ let $g_x \in G$ be an element with $\alpha_H(\pi(g_x)) = x$.
  Similarly, for each $y \in Y$ let $g_y \in G$ be the element $g_y = \iota(\alpha_N^{-1}(y))$.
  Then $G$ is generated by $g_z$, $z \in Z$.

  Given a relator $r \in T \subseteq (X \cup X^{-1})^\ast$, we can evaluate $r$ at $(g_x)_{x \in X}$ to obtain an element $g \in G$.
  Since $\pi(g) = 1$ by construction, we can determine $n \in N$ with $\iota(n) = g$.
  We set $w_r = \alpha_N(n)^{-1} r \in (Z \cup Z^{-1})^\ast$.
  Now let $x \in X$ and $y \in Y$. Since $\iota(N)$ is normal in $G$, we have
  $g_x g_y g_x^{-1} \in \iota(N)$ and we can find $w_{x, y} \in (Y \cup Y^{-1})^\ast$ with
  $\iota(\alpha_N^{-1}(w_{x, y})) = g_x g_y g_x^{-1}$.
  Let $V = \{ w_r, s, w_{x, y}^{-1}xyx^{-1} \mid x \in X, y \in Y, r \in T, s \in U \}$.
  It follows from~\cite[Proposition 2.55]{MR2129747} that the finitely presented group $P_G = \langle Z \mid V \rangle$ is isomorphic to $G$, an explicit isomorphism being induced by $x \mapsto g_x, y \mapsto g_y$.
  Clearly $V$ and thus $P_G$ can be constructed in polynomial time from the presentations of $N$ and $H$ as well as the maps $\iota$ and $\pi$.
  Note that the map $\iota' \colon P_N \to P_G$ is induced by the inclusion $X \subseteq Z$ and $\pi' \colon P_G \to P_N$ is induced by partially evaluating a word in $Z = X \sqcup Y$ by substituting $y \mapsto 1$ for $y \in Y$. In particular both maps can be evaluated in polynomial time.
  
  It remains to describe the discrete logarithm and the exponential map of this presentation of $G$.
  For the exponential map $\alpha_G^{-1} \colon P_G \to G$, let $w \in (Z \cup Z^{-1})^*$ be a word in $Z$. As $Z$ is the disjoint union of $X$ and $Y$, we can evaluate $w$ at $x \mapsto g_x$ and $y \mapsto g_y$ for $x \in X$ and $y \in Y$, yielding the element $\alpha_G^{-1}(w) \in G$.

  For the discrete logarithm $\alpha_G \colon G \to P_G$, consider an element $g \in G$.
  Then $\alpha_H(\pi(g)) \in P_H$ and we can find $w' \in (X \cup X^{-1})^\ast$ with $\pi'(w') = \alpha_H(\pi(g))$. 
  Considering $w'$ as a word in $(Z \cup Z^{-1})^*$, we can apply the exponential map to obtain $g' = \alpha_G^{-1}(w') \in G$, which by construction is an element of $\ker(\pi)$. Hence we can find $n \in N$ with $g' = \iota(n)$ for some $n \in N$.
  Then $\alpha_G(g) = \iota'(\alpha_N(n))^{-1} w'$.
  Note that both the discrete logarithm and the exponential map can be evaluated in polynomial time. The only operations performed with words are concatenation or evaluation at elements of black box groups.
\end{proof}

\begin{example}\label{example:gln}
  Let $k$ be a finite field, let $n \in \Z_{\geq 0}$ and let $\GL_n(k)$ be the general linear group over $k$ considered as a black box group. The determinant map $\det \colon \GL_n(k) \to k^\times$ induces an exact sequence
  \[ 1 \longrightarrow k^\times \longrightarrow \GL_n(k) \longrightarrow \SL_n(k) \longrightarrow 1 \]
  of black box groups. By Example~\ref{example:blackbox} and Lemma~\ref{lem:slnk}, determining effective presentations of $k^\times$ and $\SL_n(k)$ reduces in polynomial time to \textsc{PrimitiveRoots} and \textsc{DiscreteLogarithm}.
  Hence by Theorem~\ref{thm:patching}, the same is true for $\GL_n(k)$.
\end{example}

For later use, we also record the following nested version of Theorem~\ref{thm:patching}.

\begin{cor}\label{cor:patchingmany}
  Given black box groups $G_0, \dotsc,G_l$, $A_1,\dotsc,A_l$ with
  \begin{itemize}
    \item
    $G_0 = \{ 1\}$,
    \item
    for each $1 \leq i \leq l$ an exact sequence
    \[ 1 \to G_{i-1} \to G_i \to A_{i} \to 1 \]
    of black box groups,
    \item
    effective presentations for the $A_i$ with $\leq n$ generators and $\leq n^2$ relators,
  \end{itemize}
  one can construct a presentation of $G_l$ in polynomial time, and the presentation has at most $2l^2n^2$ relators and $ln$ generators.
\end{cor}

\begin{proof}
  This follows inductively from~\cref{thm:patching}.
  Denote by $s_i$ the number of generators and by $r_i$ the number of relators of the presentation of $A_i$.
  Then $s_i = s_{i-1} + n \leq i\cdot n$ and $r_i = r_{i-1} + n^2 + s_{i-1} n \leq r_{i-1} + n^2 + i \cdot n$ which shows that $s_l \leq l\cdot n$ and $r_l \leq l n^2 + l^2 n \leq 2l^2 n^2$.
\end{proof}

\section{Presentations and generators of unit groups}\label{sec:unit}

We now describe the computation of generators and effective presentations of unit groups of a finite rings,
beginning with unipotent units.

\subsection{Unipotent units.}

\begin{lemma}\label{lem:nilbasecase}
  There exists an algorithm, that given a finite ring $R$ and a nilpotent two-sided ideal $I$ of $R$, determines in polynomial time
  an effective presentation of the black box group $(1 + I)/(1 + I^2)$ with at most $r$ generators and $r^2$ relators, where $r = \rank(R^+)$, the minimal number of generators $R^+$.
\end{lemma}

\begin{proof}
  Note that the map $(1 + I)/(1 + I^2) \to I/I^2, \overline a \mapsto \overline {1 - a}$ is a well-defined isomorphism.
  We can compute the ideal $I^2$ as well as the factor group $I/I^2$ (of additive groups) as an abelian group in polynomial time.
  Thus an effective presentation can be found as claimed by~\cref{example:blackbox}~(i), since the rank of $I/I^2$ is bounded by $\rank(I) \leq \rank(R^+)$.
\end{proof}

\begin{prop}\label{prop:unipotent}
  There exists a polynomial-time algorithm, that given a finite ring $R$ and a nilpotent two-sided ideal $I$ of $R$, determines an effective presentation of $1 + I$.
\end{prop}

\begin{proof}
  Let $n \in \Z_{\geq 1}$ be minimal such that $I^n = \{0\}$. Let $k = \lceil \log_2(n) \rceil$ and let $r = \rank(R^+)$.
  Setting $I_i := I^{2^{k - i}}$, for each $0 \leq i \leq k$ we have an exact sequence
  \[ 1 \to 1 + I_{i - 1} \to 1 + I_i \to (1 + I_{i})/(1 + I_{i - 1}) \to 1.\]
  Note that $1 + I_0 = \{ 1 \}$ since $I^{2^k} = \{0\}$.
  By Lemma~\ref{lem:nilbasecase} for each $0 \leq i \leq k$ we can determine effective presentations of $(1 + I_{i})/(1 + I_{i - 1})$ in polynomial time and these presentations have at most $r$ generators and $r^2$ relators.
  Applying~\cref{cor:patchingmany} shows that we can determine an effective presentation of $1 + I$ in polynomial time.
  Note that as $n \leq \log_2(\lvert R \rvert)$ (the maximal length of a series of proper subgroups), $n$ is polynomial in the length of $R$ and the claim about the complexity follows.
\end{proof}

\subsection{The semisimple case}

We now consider the problem of computing unit groups of semisimple finite rings,
beginning with $p$-rings, that is, rings whose order is the power of a rational prime $p$.
Note that in contrast to the unipotent case, we reduce the problem to classical number theoretic questions.

\begin{lemma}\label{prop:semisimple}
  For finite semisimple $p$-rings,
  \begin{enumerate}
    \item
      the problem of determining generators of the unit group reduces in probabilistic polynomial time to \Pprimroots, and
    \item
      the problem of determining an effective presentation of the unit group reduces in probabilistic polynomial time to \Pprimroots{} and \Pdisclog.
  \end{enumerate}
\end{lemma}

\begin{proof}
  A finite semisimple $p$-ring $R$ is the same as a semisimple $\F_p$-algebra.
  In particular, the celebrated theorems of Wedderburn and Wedderburn--Artin (see for example~\cite[(3.5) \& (13.1)]{MR1838439}) imply that there exists an isomorphism
  \[ R \to \mat_{n_1}(k_1) \times \dotsb \times \mat_{n_r}(k_r) \] 
  with $n_1,\dotsc,n_r \in \Z_{\geq 1}$ and $k_1,\dotsc,k_r$ finite fields of characteristic $p$.
  Such an isomorphism can be explicitly determined in probabilistic polynomial time
  by~\cite[1.5 B]{10.1145/22145.22162} and~\cite[Theorem 6.2]{10.1145/28395.28438}.
  Thus we can assume that $R = \mat_n(k)$ is a matrix ring over a finite field $k$ and that
  $n$ as well as the length of $k$ are polynomial in the length of $R$.

  (i): Given an oracle to determine primitive roots of finite fields, generators of $\GL_n(k)$ can be obtained using a weaker version of \cref{example:gln}. Alternatively, in~\cite{Taylor1987} explicit two element generating sets of $\GL_n(k)$ are constructed (again using an oracle for primitive roots).

  (ii): This follows directly from~\cref{example:gln}.
\end{proof}

\subsection{The general case.}

We can now combine the unipotent and semisimple case, beginning again with the case of $p$-rings.

\begin{prop}\label{prop:ppower}
  For finite $p$-rings,
  \begin{enumerate}
    \item
      the problem of computing generators of the unit group reduces in probabilistic polynomial time to \Pprimroots, and
    \item
      the problem of computing an effective presentation of the unit group reduces in probabilistic polynomial time to \Pprimroots{} and \Pdisclog.
  \end{enumerate}
\end{prop}

\begin{proof}
  Let $R$ be a finite $p$-ring, and let $J$ be the Jacobson radical of $R$.
  Since $pR \subseteq J$, we can determine $J$ by lifting the Jacobson radical of $R/pR$ along the natural projection $R \to R/pR$.
  As $R/pR$ is an $\F_p$-algebra, the radical can be determined in probabilistic polynomial time by~\cite[1.5 A]{10.1145/22145.22162}.
  For the unipotent units $1 + J$ we can determine an effective presentation in polynomial time by~\cref{prop:unipotent}, which includes generators of $1 + J$.
  Consider now the exact sequence
  \[ 1 \to 1 + J \to R^\times \to (R/J)^\times \to 1. \]
  As $R/J$ is semisimple, (i) follows from~\cref{prop:semisimple}~(i), and (ii) from~\cref{prop:semisimple} together with~\cref{thm:patching}.
\end{proof}

\begin{theorem}\label{thm:main}
  For finite rings,
  \begin{enumerate}
    \item
      the problem of determining generators of unit groups reduces in probabilistic polynomial time to \Pfactor, and
    \item 
      the problem of determining effective presentations of unit groups reduces in probabilistic polynomial time to \Pfactor{} and \Pdisclog.
  \end{enumerate}
\end{theorem}

\begin{proof}
  By factoring $\lvert R \rvert$, we can construct an isomorphism $R \to R_1 \times \dotsb \times R_r$, where $R_1,\dotsc,R_n$ are finite rings of coprime prime power orders.
  The result follows then from~\cref{prop:ppower} noting that \Pprimroots{} reduces in probabilistic polynomial time to \Pfactor.
\end{proof}

\section{Abelianization \& $\K1$}\label{sec:abinv}

In this section we consider certain abelian invariants of finite rings, which are functors from the category of finite rings to the category of finite abelian groups.

\subsection{Abelianization}

For a group $G$, denote by $G^{\mathrm{ab}} = G/[G, G]$ the abelianization (or maximal abelian quotient) of $G$.
For a finite ring $R$, we consider the problem of computing $\ab{R}$, a problem which we first have to make precise.
Considering $R^\times$ as a black box group, by determining $\ab{R}$ we mean determining a finite abelian group $A$
together with a surjective homomorphism $R^\times \to A$ with kernel $[R^\times, R^\times]$ and for which images and preimages can be determined in probabilistic polynomial time (in particular $A \cong \ab{R}$).

\begin{prop}\label{thm:unitab}
  The problem of
  computing the abelianization $\ab{R}$ for a given finite ring $R$ 
  reduces in polynomial time to the problem of computing an effective presentation of $R^\times$.%
\end{prop}

\begin{proof}
  By~\cref{thm:main}~(ii) we can find a presentation of $R^\times$.
  For a finitely presented group $G$, a presentation for the abelianization $\ab{G}$ can be read off from the relators; see ~\cite[Proposition 2.68]{MR2129747} or~\cite[Sec. 11.2]{MR1267733}.
  Determining this as an abelian group (using the encoding from~\Cref{sec:prelim}) is then just a Smith normal form computation,
  which can be done in polynomial time by~\cite{MR573842}.
\end{proof}

\subsection{Computation of $\K1$}

We recall the construction of $\K1$ for a ring $R$ (for details see example~\cite{MR629979}).
For $n \geq 1$, the general linear group $\GL_n(R)$ embeds into $\GL_{n+1}(R)$ via
\[ A \mapsto \begin{pmatrix} A & 0 \\
0 & 1 \end{pmatrix}. \]
We obtain the infinite general linear group $\GL(R) = \varinjlim \GL_n(R)$ and the group $\K1(R) = \GL(R)^{\mathrm{ab}}$.
Note that there is a canonical group homomorphism
\[ c : R^\times \to \GL_1(R) \to \GL(R) \to \K1(R) \] 
which is surjective if $R$ is a finite ring by \cite[Theorem 4.2(b)]{MR174604}.
Considering $R^\times$ as a black box group, by determining $\K1(R)$ we mean determining a finite abelian group $A$
together with a surjective homomorphism $R^\times \to A$ with kernel $\ker(c)$ and for which images and preimages can be determined in polynomial time (in particular $A \cong \K1(R)$).

\begin{prop}\label{thm:k1}
  The problem of computing $\K1$ of finite rings ($p$-rings)
  reduces in probabilistic polynomial time to computing abelianizations of unit groups of finite rings ($p$-rings). %
\end{prop}

\begin{proof}
  Consider the canonical map
  \[ f \colon  R^\times = \GL_1(R) \to \GL_3(R) \to \GL_3(R)^{\mathrm{ab}}, \, a \mapsto \overline{\begin{pmatrix}
  a & 0 & 0 \\ 0 & 1 & 0 \\ 0 & 0 & 1 \end{pmatrix}}. \] 
  By~\cite[Prop. 52]{MR629979} we know that $f$ is surjective with
  \[ {R^\times}{}/{}{\ker(f)} \cong \GL_3(R)^{\mathrm{ab}} \cong \K1(R). \] 
  Now $S = \mat_3(R)$ is again a finite ring of length a polynomial in the length of $R$.
  By assumption we can determine the abelianization $(S^{\times})^{\mathrm{ab}} = \GL_3(R)^{\mathrm{ab}}$
  and $(R^{\times})^{\mathrm{ab}}$.
  By lifting elements from $(R^{\times})^{\mathrm{ab}}$ to $R^\times$ and mapping them under $f$ %
  we can determine explicitly the induced map
$\overline f \colon (R^{\times})^{\mathrm{ab}} \to (S^{\times})^{\mathrm{ab}}$ of abelian groups.
  The claim follows by observing that $(R^{\times})^{\mathrm{ab}}/\ker(\overline{f}) \cong R^\times/\ker(f) \cong \K1(R)$.
  For $p$-rings the result follows from the fact that $S$ is a $p$-ring if $R$ is a $p$-ring.
\end{proof}

\section{Relations between computational problems}\label{sec:compproblems}

In previous sections, we have reduced the computational problems related to unit groups of finite rings to classical problems in algorithmic number theory. We now show that (except for the problem of computing generators), these are in fact equivalent problems.

\begin{prop}\label{lem:comringfactor}
  The following computational problems are probabilistic polynomial-time reducible to each other:
  \begin{enumerate}
    \item
      \Pfactor{},
    \item
      computing the order of unit groups of finite rings,
    \item
      computing the order of unit groups of finite commutative rings,
    \item
      computing the order of unit groups of rings $\Z/n\Z$, $n \neq 0$.
  \end{enumerate}
\end{prop}

\begin{proof}
  (i) $\Rightarrow$ (ii): Assume that we have an oracle for factoring integers and let $R$ be a finite ring.
  We can assume that $R$ is a $p$-ring for some prime $p$.
  In probabilistic polynomial time we can determine the Jacobson radical $J = \jac(R)$; see~\cite[1.5 A]{10.1145/22145.22162},
  as well as finite extensions $k_1,\dotsc,k_r$ of $\F_p$ and integers $n_1,\dotsc,n_r$ such that
  \[ R/J \cong \mat_{n_1}(k_1) \times \dotsb \times \mat_{n_r}(k_r). \] 
  Thus we can determine $\lvert R^\times \rvert = \lvert 1 + J \rvert \cdot \lvert (R/J)^\times \rvert = \lvert J \rvert \cdot \lvert \GL_{n_1}(k_1) \rvert \dotsm  \lvert \GL_{n_r}(k_r) \rvert$.

  (ii) $\Rightarrow$ (iii), (iii) $\Rightarrow$ (iv): Clear.

  (iv) $\Rightarrow$ (i): This follows from the fact, that there exists a probabilistic polynomial-time algorithm that given $\varphi(n) = \lvert (\Z/n\Z)^\times \rvert$ determines a factorization of $n$; see~\cite[Sec. 10.4]{MR2488898} and~\cite{MR480295}.
\end{proof}

\begin{theorem}\label{thm:mainmain}
  The following computational problems are probabilistic polynomial-time reducible to each other:
  \begin{enumerate}
    \item
      \Pfactor{} and \Pdisclog,
    \item
      computing effective presentations of unit groups of finite rings,
    \item
      computing effective presentations of unit groups of finite commutative rings,
    \item
      computing effective presentations of unit groups of rings $\Z/n\Z$, $n \neq 0$, and finite fields,
    \item
      computing abelianizations of unit groups of finite rings,
    \item
      computing $\K1$ of finite rings.
  \end{enumerate}
\end{theorem}

\begin{proof}
  (i) $\Rightarrow$ (ii): \cref{thm:main}. (ii) $\Rightarrow$ (iii) and (iii) $\Rightarrow$ (iv) are clear.
  (ii) $\Rightarrow$ (v): \cref{thm:unitab}.

  (v) $\Rightarrow$ (vi): \cref{thm:k1}.

  (vi) $\Rightarrow$ (iii): This follows from the fact that for a finite commutative ring $R$ one has that the canonical map $R^\times \to \ab{R} \to \K1(R)$ is an isomorphism (see~\cite[Theorem 3.6(a)]{MR267009}).

  (iv) $\Rightarrow$ (i):
  We can solve \Pfactor{} by~\cref{lem:comringfactor}.
  Now let $k$ be a finite field and $\alpha, \beta \in k^\times$ with $\alpha \in \langle \beta \rangle$.
  By assumption we can determine an effective presentation of $k^\times$, from which we can determine the abelianization $k^\times \cong \ab{k}$, represented as an abelian group.
  Since it is straightforward to solve the discrete logarithm in our encoding of finite abelian groups, the claim follows.
\end{proof}

\begin{remark}\label{remark:nofactor}
  Note that restricting to $p$-rings, one can obtain an analogous result as in~\Cref{thm:main}, where (i) and (iv) are replaced by
  \begin{enumerate}
    \item[(i)$'$]
      \Pprimroots{} and \Pdisclog,
    \item[(iv)$'$]
      computing effective presentations of unit groups of finite fields.
  \end{enumerate}
  This is obtained from \Cref{prop:ppower} and the observation
  that the proof of ``(iv) $\Rightarrow$ (i)'' of~\Cref{thm:mainmain} shows that once one can determine presentations of unit groups of finite fields, one can solve \Pprimroots{} and \Pdisclog.
\end{remark}

\section{Applications}\label{sec:applications}

\subsection{Finite-dimensional algebras over finite fields}

We present applications of \Cref{thm:main} to various classes of finite rings, beginning with finite-dimensional algebras over finite fields.
Recall that if $A$ is a $d$-dimensional algebra over $\F_q$, the length of $A$ is $\log(\lvert A \rvert)^{O(1)} = (d \log(q))^{O(1)}$.
By applying~\Cref{thm:main} and~\Cref{remark:nofactor} we immediately obtain:

\begin{cor}\label{cor:alg}
  The problem of computing an effective presentation (or abelianization, or $\K1$ respectively) of $A^\times$ for a given finite-dimensional algebra $A$ over a finite field reduces in probabilistic polynomial time to \Pprimroots{} and \Pdisclog{}.
\end{cor}

\subsection{Finite quotient rings of arithmetic orders}

The next application is motivated by computational algebraic number theory and noncommutative generalizations.
Let $R$ be either $\Z$ or $\F_q[X]$ with $q$ a prime power.
An \textit{arithmetic $R$-order} is a torsion-free finite $R$-algebra.
(Note that the standard encoding of $R$ extends to an encoding of arithmetic orders and also to ideals of arithmetic orders.)
Examples of arithmetic orders are $\Z$-orders in number fields or finite-dimensional $\Q$-algebras, and integral group rings $\Z[G]$, where $G$ is a finite group.

Let $\Lambda$ be an arithmetic $R$-order. A two-sided ideal $I$ of $\Lambda$ is of \textit{full rank}, if $\Lambda/I$ is finite.
In this case the residue class ring $\Lambda/I$ is a finite ring with length polynomial in the lengths of $\Lambda$ and $I$ respectively.
Thus by~\cref{thm:main} we obtain the following result, which includes and generalizes previous results on arithmetic orders, as mentioned in the introduction.

\begin{cor}\label{cor:arithmetic}
  The following computational problems are probabilistic polynomial-time reducible to each other:
  \begin{enumerate}
    \item
      \Pfactor{} and \Pdisclog{},
    \item
      computing effective presentations of unit groups of finite quotient rings of arithmetic orders,
    \item
      computing abelianizations of unit groups of finite quotient rings of arithmetic orders,
    \item
      computing $\K1$ of finite quotient rings of arithmetic orders,
    \item
      computing effective presentations of finite quotient rings of (maximal) $\Z$-orders in number fields.
  \end{enumerate}
\end{cor}

\begin{proof}
  By \Cref{thm:mainmain} it suffices to show that for finite quotients of $\Z$ and finite fields $k$, one can construct a ring isomorphic quotient of a $\Z$-order in a number field, together with an isomorphism in probabilistic polynomial time.
  As $\Z$ itself is the maximal $\Z$-order of the number field $\Q$, it suffices to consider finite fields. Also, we can solve $\Pfactor$ (as in the proof of \Cref{thm:mainmain}).
  Let $f \in \F_p[X]$ be a monic defining polynomial, and let $g \in \Z[X]$ be a monic lift with nonnegative coefficients bounded by $p$.
  Then $g$ is irreducible with length polynomial in the length of $k$, and the $\Z$-order $\Lambda = \Z[X]/(g)$ of the number field defined by $g$ satisfies $\Lambda/p \Lambda \cong k$. Moreover, such an isomorphism can be found in polynomial time by~\cite{MR1052099}.
  As the computation of a maximal order reduces in probabilistic polynomial time to \Pfactor{} (see~\cite{Lenstra1992}), we can also assume that $\Lambda$ is maximal.
\end{proof}

\subsection{Group rings}

As a final application we apply our result to the problem of computing unit groups of group rings $R[G]$, where $G$ is a finite group and $R$ is itself a finite ring, for example $\F_q[G]$. Note that here we consider $G$ as being encoded via a Cayley table and thus of length $\lvert G \rvert^{O(1)}$. Therefore $R[G]$ is encoded with length $(\lvert G \rvert \log(\lvert R \rvert))^{O(1)}$.

\begin{cor}
  The problem of computing an effective presentation (or abelianization, or $\K1$ respectively) of $R[G]$ for a given finite group $G$ and finite ring $R$, reduces in probabilistic polynomial time to \Pfactor{} and \Pdisclog{}.
\end{cor}

Note that the same results hold if one replaces $G$ by a monoid and considers the monoid ring $R[G]$.

\section{Implementation \& examples}\label{sec:computations}

The algorithms presented in the previous sections have been implemented in the computer algebra systems \textsc{Hecke}~\cite{FHHJ2017} and~\textsc{OSCAR}~\cite{OSCAR, OSCAR-book} and will be available in the next official release.
Based upon this implementation, we present numerical examples and compare the algorithm with an approach from~\cite{bley2024determinationstablyfreecancellation} for quotients of arithmetic orders.

\subsection{First K-group for modular group rings of small groups}

In~\cite{MR2220078}, Magurn determines $\K1(\F_p[G])$ for different finite groups $G$, including dihedral groups of order $2m$, $m = 2^r \geq 4$ (a result which was stated in Keating~\cite{MR689372}), the alternating groups $\mathfrak{A}_4$, $\mathfrak{A}_5$, the symmetric group $\mathfrak{S}_4$, and for arbitrary groups of the form $\mathrm{C}_2^m \times G$.
We have used our algorithm to determine $\ab{\F_2[G]}$ and $\K1(\F_2[G])$ for all nonabelian groups of even order up to $32$. (We have also computed $\F_p[G]^\times$, but there is no convenient way to present the results.)
The results are presented in Table~\ref{tab:results}.
For a finite group $G$ with ID $(i, j)$ in the library of small groups~\cite{MR1935567}, we display unique invariants of the finite abelian groups $\ab{\F_2[G]}$ and $\ab{\F_2[G]}/\K1(\F_2[G])$.
A group $(\Z/d_1\Z)^{n_1} \times \dotsb \times (\Z/d_r\Z)^{n_r}$ with $d_1 \mid d_2 \dotsm \mid d_r$ is represented by a list $[d_1^{n_1},\dotsc,d_r^{n_r}]$ of positive integers together with exponents, where we omit $n_i$ if the value is equal to $1$.
The results are consistent with
\begin{itemize}
  \item
    the results from~\cite{MR2220078},
  \item
    the fact that for a $p$-group the ring $\F_p[G]$ is local (see \cite[(5.24) Theorem]{curtisandreiner_vol1}), so that $\K1(\F_2[G]) \cong \ab{\F_2[G]}$ by~\cite[Theorem 3.6(b)]{MR267009},
  \item
    the following formula~\cite[Corollary]{MR2220078}: If $G$ is a finite group with $c$ conjugacy classes, then
    \[ \K1(\F_2[\mathrm{C}_2^n \times G]) \cong (\Z/2\Z)^{c(2^n - 1)} \times \K1(\F_2[G]). \] 
\end{itemize}

\begin{table}[ht!] 
  \caption{First K-group for modular group rings of small groups.}
  \small
  \label{tab:results}
  \begin{tabular}{ccll}\toprule
    $G$ & ID & $\ab{R}$ & $\ab{R}/\K1$  \\\midrule
$\mathfrak{S}_3$                      & $(6,1)$ & $[2^2]$ & $[2]$ \\
$\mathrm{D}_8$                        & $(8,3)$ & $[2^2,4]$ & $[]$ \\
$\mathrm{Q}_8$                        & $(8,4)$ & $[2^2,4]$ & $[]$ \\
$\mathrm{D}_{10}$                     & $(10,1)$ & $[6]$ & $[]$ \\
$\mathrm{Q}_{12}$                     & $(12,1)$ & $[2^3,4]$ & $[2]$ \\
$\mathfrak{A}_4$                      & $(12,3)$ & $[6]$ & $[]$ \\
$\mathrm{D}_{12}$                     & $(12,4)$ & $[2^5]$ & $[2]$ \\
$\mathrm{D}_{14}$                    & $(14,1)$ & $[14]$ & $[]$ \\
$\mathrm{C}_2^2\rtimes \mathrm{C}_4$ & $(16,3)$ & $[2^5,4^2]$ & $[]$ \\
$\mathrm{C}_4\rtimes \mathrm{C}_4$ & $(16,4)$ & $[2^5,4^2]$ & $[]$ \\
$\mathrm{M}_4(2)$ & $(16,6)$ & $[2^3,4^3]$ & $[]$ \\
$\mathrm{D}_{16}$ & $(16,7)$ & $[2^3,8]$ & $[]$ \\
$\mathrm{SD}_{16}$ & $(16,8)$ & $[2^3,8]$ & $[]$ \\
$\mathrm{Q}_{16}$ & $(16,9)$ & $[2^3,8]$ & $[]$ \\
$\mathrm{C}_2 \times \mathrm{D}_{8}$ & $(16,11)$ & $[2^7,4]$ & $[]$ \\
$\mathrm{C}_2 \times \mathrm{Q}_{8}$ & $(16,12)$ & $[2^7,4]$ & $[]$ \\
$\mathrm{C}_4 \circ \mathrm{D}_{8}$ & $(16,13)$ & $[2^7,4]$ & $[]$ \\
$\mathrm{D}_{18}$ & $(18,1)$ & $[2,14]$ & $[2]$ \\
$\mathrm{\mathrm{C}_3} \times \mathfrak{S}_3$ & $(18,3)$ & $[2^2,6^2]$ & $[2]$ \\
$\mathrm{C}_3\rtimes \mathfrak{S}_3$ & $(18,4)$ & $[2^5]$ & $[2^4]$ \\
$\mathrm{Q}_{20}$ & $(20,1)$ & $[2^3,12]$ & $[]$ \\
$\mathrm{F}_5$ & $(20,3)$ & $[2,4]$ & $[]$ \\
$\mathrm{D}_{20}$ & $(20,4)$ & $[2^4,6]$ & $[]$ \\
$\mathrm{D}_{22}$ & $(22,1)$ & $[62]$ & $[]$ \\
$\mathrm{C}_3\rtimes \mathrm{C}_8$ & $(24,1)$ & $[2^4,4^2,8]$ & $[2]$ \\
$\SL_2(\F_3)$ & $(24,3)$ & $[2^2,12]$ & $[]$ \\
$\mathrm{Q}_{24}$ & $(24,4)$ & $[2^4,4^2]$ & $[2]$ \\
$\mathrm{C}_4 \times \mathfrak{S}_3$ & $(24,5)$ & $[2^7,4^2]$ & $[2]$ \\
$\mathrm{D}_{24}$ & $(24,6)$ & $[2^4,4^2]$ & $[2]$ \\
$\mathrm{C}_2\times \mathrm{Q}_{12} $ & $(24,7)$ & $[2^9,4]$ & $[2]$ \\
& $(24,8)$ & $[2^6,4]$ & $[2]$ \\
$\mathrm{C}_3\times \mathrm{D}_8$ & $(24,10)$ & $[2^6,4^2,12]$ & $[]$ \\
$\mathrm{C}_3\times \mathrm{Q}_8$ & $(24,11)$ & $[2^6,4^2,12]$ & $[]$ \\
$\mathfrak{S}_4$ & $(24,12)$ & $[2^2,4]$ & $[2]$ \\
$\mathrm{C}_2\times \mathfrak{A}_4$ & $(24,13)$ & $[2^4,6]$ & $[]$ \\
$\mathrm{C}_2^2\times \mathfrak{S}_3$ & $(24,14)$ & $[2^{11}]$ & $[2]$ \\
$\mathrm{D}_{26}$ & $(26,1)$ & $[126]$ & $[126]$ \\
$\mathrm{Q}_{28}$ & $(28,1)$ & $[2^{4},28]$ & $[]$ \\
$\mathrm{D}_{28}$ & $(28,3)$ & $[2^{5},14]$ & $[]$ \\
$\mathrm{C}_5 \times \mathfrak{S}_3$ & $(30,1)$ & $[2^{4},30^{2}]$ & $[2]$ \\
$\mathrm{C}_3 \times \mathrm{D}_{10}$ & $(30,2)$ & $[3,6^{3}]$ & $[]$ \\
$\mathrm{D}_{30}$ & $(30,3)$ & $[6,30]$ & $[6]$ \\
\phantom{x} \\
\phantom{x} \\
\bottomrule
  \end{tabular}
  \hfill
  \begin{tabular}{cccc}\toprule
    $G$ & ID & $\ab{R}$ & $\ab{R}/\K1$ \\\midrule
    & $(32,2)$ & $[2^{13},4^{3}]$ & $[]$ \\
& $(32,4)$ & $[2^{7},4^{6}]$ & $[]$ \\
$\mathrm{C}_2^2\rtimes \mathrm{C}_8$ & $(32,5)$ & $[2^{10},4^{3},8]$ & $[]$ \\
& $(32,6)$ & $[2^{4},4^{3}]$ & $[]$ \\
& $(32,7)$ & $[2^{4},4^{3}]$ & $[]$ \\
& $(32,8)$ & $[2^{4},4^{3}]$ & $[]$ \\
& $(32,9)$ & $[2^{8},4,8]$ & $[]$ \\
& $(32,10)$ & $[2^{8},4,8]$ & $[]$ \\
$\mathrm{C}_4\wr \mathrm{C}_2$ & $(32,11)$ & $[2^{6},4^{2},8]$ & $[]$ \\
$\mathrm{C}_4\rtimes \mathrm{C}_8$ & $(32,12)$ & $[2^{10},4^{3},8]$ & $[]$ \\
& $(32,13)$ & $[2^{8},4,8]$ & $[]$ \\
& $(32,14)$ & $[2^{8},4,8]$ & $[]$ \\
& $(32,15)$ & $[2^{6},4^{2},8]$ & $[]$ \\
$\mathrm{M}_5(2)$ & $(32,17)$ & $[2^{7},4^{3},8^{2}]$ & $[]$ \\
$\mathrm{D}_{32}$ & $(32,18)$ & $[2^{4},4,16]$ & $[]$ \\
$\mathrm{SD}_{32}$ & $(32,19)$ & $[2^{4},4,16]$ & $[]$ \\
$\mathrm{Q}_{32}$ & $(32,20)$ & $[2^{4},4,16]$ & $[]$ \\
 & $(32,22)$ & $[2^{15},4^{2}]$ & $[]$ \\
 & $(32,23)$ & $[2^{15},4^{2}]$ & $[]$ \\
& $(32,24)$ & $[2^{13},4^{3}]$ & $[]$ \\
$\mathrm{C}_4\times \mathrm{D}_4$ & $(32,25)$ & $[2^{13},4^{3}]$ & $[]$ \\
$\mathrm{C}_4\times \mathrm{Q}_8$ & $(32,26)$ & $[2^{13},4^{3}]$ & $[]$ \\
$\mathrm{C}_2^2\wr \mathrm{C}_2$ & $(32,27)$ & $[2^{7},4^{3}]$ & $[]$ \\
& $(32,28)$ & $[2^{9},4^{2}]$ & $[]$ \\
& $(32,29)$ & $[2^{9},4^{2}]$ & $[]$ \\
& $(32,30)$ & $[2^{7},4^{3}]$ & $[]$ \\
& $(32,31)$ & $[2^{7},4^{3}]$ & $[]$ \\
& $(32,32)$ & $[2^{7},4^{3}]$ & $[]$ \\
& $(32,33)$ & $[2^{7},4^{3}]$ & $[]$ \\
& $(32,34)$ & $[2^{7},4^{3}]$ & $[]$ \\
& $(32,35)$ & $[2^{7},4^{3}]$ & $[]$ \\
$\mathrm{C}_2\times \mathrm{M}_4(2)$ & $(32,37)$ & $[2^{13},4^{3}]$ & $[]$ \\
& $(32,38)$ & $[2^{13},4^{3}]$ & $[]$ \\
$\mathrm{C}_2\times \mathrm{D}_8$ & $(32,39)$ & $[2^{10},8]$ & $[]$ \\
$\mathrm{C}_2\times {\rm SD}_{16}$ & $(32,40)$ & $[2^{10},8]$ & $[]$ \\
$\mathrm{C}_2\times \mathrm{Q}_{16}$ & $(32,41)$ & $[2^{10},8]$ & $[]$ \\
& $(32,42)$ & $[2^{10},8]$ & $[]$ \\
& $(32,43)$ & $[2^{7},8]$ & $[]$ \\
& $(32,44)$ & $[2^{7},8]$ & $[]$ \\
$\mathrm{C}_2^2\times \mathrm{D}_4$ & $(32,46)$ & $[2^{17},4]$ & $[]$ \\
$\mathrm{C}_2^2\times \mathrm{Q}_8$ & $(32,47)$ & $[2^{17},4]$ & $[]$ \\
& $(32,48)$ & $[2^{17},4]$ & $[]$ \\
& $(32,49)$ & $[2^{14},4]$ & $[]$ \\
& $(32,50)$ & $[2^{14},4]$ & $[]$ \\
\bottomrule
  \end{tabular}
\end{table}

\subsection{Examples from noncommutative arithmetic orders}

We consider computations of abelianizations of unit groups originating from~\cite{bley2024determinationstablyfreecancellation} in the context of the stably free cancellation property of integral group rings and $\Z$-orders in general.
The finite rings we consider arise as follows.
For a $\Z$-order $\Lambda$ of a semisimple $\Q$-algebra $A$, any non-trivial decomposition $A \cong A_1 \times A_2$ induces a fiber product involving a finite quotient ring $R$ of $\Lambda$.
Of particular interest are integral group rings $\Lambda = \Z[G]$, where $G$ is related to binary polyhedral groups.
In~\cite{bley2024determinationstablyfreecancellation} the problem of determining $R^\times$ and $\ab{R}$ was solved by constructing an explicit faithful permutation representation $\Aut_{\Z}(R^+) \to \mathfrak{S}_d$, or (if $\operatorname{char}(R) = p$) a faithful linear representation $\Aut_{\Z}(R^+) \to \GL_n(\F_p)$.
Together with the algorithms of~\cite{MR2282916,bley2024determinationstablyfreecancellation} for determining generators of $R^\times$, one obtains a ``realization'' of $R^\times \to \Aut_{\Z}(R^+)$ (considered as a black box group) as a subgroup of a direct product of permutation groups or linear groups over finite fields, from which $\ab{R}$ can be determined.
The algorithm of loc. cit. has been implemented using both \textsc{OSCAR} and \textsc{Magma}~\cite{Magma}.
Note that the complexity of this algorithm was not analyzed, but as the minimal degree of a faithful permutation presentation appears to grow fast, it is plausible that the algorithm is not polynomial in the length of $R$.

In the following, for $n \geq 2$ we denote by $\mathrm{Q}_{4n}$ the generalized quaternion group (also known as the dicyclic group) of order $4n$.
The examples provided in this section are also available on the author's homepage.

\begin{example}
  Let $G = \SL_2(\F_3) \times \mathrm{Q}_{12}$. We consider a finite quotient ring $R$ of $\Z[G]$ with underlying abelian group 
  \[ R^+ \cong (\Z/6\Z)^2 \times (\Z/18\Z)^2 \times (\Z/36\Z)^4 \times (\Z/72\Z)^4. \] 
  In less than a second, our algorithm finds a presentation of $R^\times$ with $24$ generators and $301$ relators and 
  \[ \ab{R} \cong (\Z/2\Z)^6 \times (\Z/6\Z)^3 \times (\Z/24\Z). \]
  The algorithm of~\cite{bley2024determinationstablyfreecancellation} finds an embedding
  \[ R^\times \to \mathfrak{S}_{648} \times \mathfrak{S}_{3840} \times \mathfrak{S}_{3840} \times \mathfrak{S}_{6480} \times \mathfrak{S}_{6480} \] 
  and determines $\ab{R}$ in $17$ seconds.
\end{example}

\begin{example}
  Let $G = \SL_2(\F_3) \times \mathrm{Q}_{20}$. We consider a finite quotient ring $R$ of $\Z[G]$ with underlying abelian group 
  \[ R^+ \cong (\Z/6\Z)^4 \times (\Z/12\Z)^2 \times (\Z/60\Z)^2 \times (\Z/120\Z)^8. \] 
  In less than a second, our algorithm finds a presentation of $R^\times$ with $47$ generators and $1655$ relators and 
  \[ \ab{R} \cong (\Z/2\Z)^8 \times (\Z/4\Z)^4 \times (\Z/12\Z) \times (\Z/24\Z). \]
  The algorithm of~\cite{bley2024determinationstablyfreecancellation} finds an embedding
  \[ R^\times \to \GL_2(\F_5) \times \GL_4(\F_5) \times \GL_4(\F_5) \times \mathfrak{S}_{16711680} \] 
  and determines $\ab{R}$ in $180$ minutes.
\end{example}

\begin{example}
  Let $G = \mathrm{Q}_8 \rtimes \SL_2(\F_3)$. We consider a finite quotient ring $R$ of $\Z[G]$ with underlying abelian group 
  \[ R^+ \cong (\Z/24\Z)^8 .\] 
  In less than a second, our algorithm finds a presentation of $R^\times$ with $31$ generators and $496$ relators and
  \[ \ab{R} \cong (\Z/2\Z)^5 \times (\Z/12\Z). \]
  The algorithm of~\cite{bley2024determinationstablyfreecancellation} finds an embedding
  \[ R^\times \to \GL_4(\F_3) \times \GL_4(\F_3) \times \mathfrak{S}_{16711680} .\] 
  and determines $\ab{R}$ in 96 minutes.
\end{example}

\bibliography{main.bib}
\bibliographystyle{amsalpha}

\end{document}